\documentclass{amsart}

\usepackage{tikz}
\usepackage{cite}
\usepackage{hyperref}
\usetikzlibrary{arrows, cd}
\tikzset{> = stealth}
\tikzcdset{arrow style = tikz}

\newtheorem{thm}{Theorem}[section]
\newtheorem{prp}[thm]{Proposition}
\newtheorem{lem}[thm]{Lemma}
\newtheorem{cor}[thm]{Corollary}
\theoremstyle{change}
\theoremstyle{nonumberplain}

\theoremstyle{definition}
\newtheorem{dfn}[thm]{Definition}

\theoremstyle{remark}

\numberwithin{equation}{section}

\begin{document}


\title[Homotopy invariance of stabilized configuration spaces]{An elementary proof of the homotopy invariance of stabilized configuration spaces}


 \author{Connor Malin}
 \address{University of Notre Dame}
 \email{cmalin@nd.edu}

\begin{abstract}
In this paper we give an elementary proof of the proper homotopy invariance of the equivariant stable homotopy type of the configuration space $F(M,k)$ for a topological manifold $M$. Our technique is to compute the Spanier-Whitehead dual of $\Sigma^\infty_+ F(M,k)$ and use the results of Spivak and Wall on normal spherical fibrations to deduce that the Spanier-Whitehead dual is a proper homotopy invariant. This stable invariance was recently proved by Knudsen using factorization homology. Aside from being elementary, our proof has the advantage that it readily extends to ``generalized configuration spaces'' which have recently undergone study.
\end{abstract}
\maketitle



\section{Introduction}

Let $M$ be an $n$-manifold. The configuration space of $k$ points in $M$, $F(M,k)$, is the collection of tuples $(x_1,\dots,x_k)$ with $x_i\neq x_j$ for $i \neq j$. There is an evident free action of $S_k$, the $k^{\operatorname{th}}$ symmetric group, on this space. The quotient $F(M,k)/S_k$ is the unordered configuration space of $k$ points in $M$. Because this action is free, if one is interested in the $S_k$-action, it suffices to study the Borel equivariant homotopy type, i.e. two $S_k$-spaces or $S_k$-spectra are Borel equivalent if there is an equivariant map that is an underlying weak homotopy equivalence of spaces or spectra. 

Configuration spaces are classic tools and objects of study in algebraic and geometric topology. Despite being simple to define, even the homology of configuration spaces is rather mysterious. For example, Bodigheimer-Cohen-Taylor show the $\mathbb{F}_p$ homology of $F(M,k)/S_k$, $M$ smooth and compact, depends only on the $\mathbb{F}_p$ homology of $M$ if $n=2k+1$ \cite[Theorem C]{BCT}, which is contrary to the even dimensional case where Zhang has shown that it depends on the cup product structure \cite[Corollary 6.9]{zhang}.

The problem of understanding the homotopy types of configuration spaces is ongoing; one might conjecture that the homotopy types of configuration spaces themselves are homotopy invariants of the $n$-manifold $M$ when we restrict to a fixed dimension $n$. This is known to be false in general. Longoni and Salvatore show that $3$-dimensional lens spaces give counterexamples \cite[Section 5]{LS}. Alternatively, standard arguments with Fadell-Neuwirth fibrations show that for a contractible $n$-manifold $M$, the homotopy type of $F(M,k)$ is $F(\mathbb{R}^n,k)$. From these examples, we already see the dependence of $F(M,k)$ on the homotopy type of $M$ is subtle. The question of the precise dependence is highly related to embedding calculus, which seeks to study manifolds via operad actions on their configuration spaces.

From now on, we assume that $M$ is a tame, topological manifold; i.e. $M$ is homeomorphic to the interior of a compact manifold with, possibly empty, boundary.

There is an easy observation one can make about homology of configuration spaces. For a pointed space $X$, let $\Delta^{\mathrm{fat}}(X^{\wedge k})$ denote the subspace of $X^{\wedge k}$ of tuples where some $x_i=x_j$ for $i \neq j$. The one point compactification $F(M,k)^+$ is homeomorphic to $(M^+)^{\wedge k} /\Delta^{\mathrm{fat}}((M^+)^{\wedge k})$, so Poincaré  duality tells us that, for oriented $M$, the homology of $F(M,k)_+$ is the cohomology of $(M^+)^{\wedge k}/\Delta^{\mathrm{fat}}((M^+)^{\wedge k})$, up to a shift of $nk$. However, the latter is evidently a proper homotopy invariant, and so the homology of $F(M,k)$ is a proper homotopy invariant.

Since homology is stable with respect to suspension, one might conjecture that the spectrum $\Sigma^\infty_+ F(M,k)$ is a proper homotopy invariant. Indeed, Aouina-Klein gave explicit bounds on the number of suspensions needed for configurations spaces of homotopy equivalent compact PL manifolds to become (nonequivariantly) equivalent \cite[Theorem A]{ak}. Recently, Knudsen proved an equivariant version of this result for all tame manifolds \cite[Theorem C]{knudsen}, implying that additionally the stable homotopy type of unordered configuration spaces is a homotopy invariant of $M^+$. Cohen-Taylor also have an unpublished sketch in the smooth case using the theory of labeled configuration spaces.

The proof techniques used in these approaches have very few commonalities. The first requires some nontrivial geometric topology and extended use of Poincaré  duality spaces, the second requires factorization homology and the theory of higher enveloping algebras for spectral lie algebras, and the last requires somewhat substantial homotopy theory.

The purpose of this paper is to give an elementary proof of the proper homotopy invariance of the spectrum $\Sigma^\infty_+ F(M,k)$.

\begin{thm}
If $M_1,M_2$ are tame, topological manifolds which are proper homotopy equivalent, then there is a Borel equivariant equivalence \[\Sigma^\infty_+ F(M_1,k) \simeq \Sigma^\infty_+ F(M_2,k).\]
\end{thm}

In addition to our approach being completely elementary, there is an advantage to our techniques. The proof works for a more general class of spaces than configuration spaces. In the last section, we prove an analogous theorem about variants of configuration spaces that appear in \cite{bm}, \cite{hai}, \cite{petersen}, \cite{vw}.

Our proof is a spectrum level refinement of the observation that the homology of $F(M,k)$ is a proper homotopy invariant. Indeed, if $M$ is a framed, compact manifold, then by Atiyah duality the Spanier-Whitehead dual of $\Sigma^\infty_+F(M,k)$ is $\Sigma^\infty M^k/\Delta^{\mathrm{fat}}(M^k)$, and one is done since the latter is a homotopy invariant. To begin our proof, in section 2, we construct a duality pairing for framed manifolds which will eventually be used to calculate the Spanier-Whitehead dual $\Sigma^\infty_+ F(M,k)^\vee$. In section 3, we review some results about bundles which will be used in the proof of the main theorem. In section 4, we briefly recall what is known about normal bundles of topological manifolds. In section 5, we prove our main theorem by identifying $\Sigma^\infty_+F(M,k)^\vee$ with a quotient of a certain relative Thom space over $(M^+)^{\wedge k}$. In section 6, we generalize the main theorem to other types of configuration spaces.

\textbf{Acknowledgements:} I would like to thank Ben Knudsen and Mark Behrens for reviewing drafts of this paper, as well as the anonymous reviewer for their detailed comments.
\section{An intrinsic duality pairing}

From now on, the manifold $M$ is assumed tame, i.e. abstractly homeomorphic to the interior of a compact topological $n$-manifold with boundary.
\\

In this section, we construct an intrinsic duality pairing for framed manifolds. Because of the tautological nature of the pairing, when there are group actions on $M$ and $S^n$, it will be easy to check when the pairing is equivariant.

For spectra $A,B$, let $F(A,B)$ denote the mapping spectrum. The only properties we need of $F(-,-)$ are that if $A,B$ are finite, then given a map $A \wedge B \rightarrow S^0$ there is a natural identification for any field $k$ of  \begin{center}$H_*((A \rightarrow F(B,S^0));k)$ with  $(H_*(A;k)\rightarrow \mathrm{Hom}(H_*(B;k),k))$
\end{center} 

\begin{dfn}
Let $X,Y,Z$ be pointed spaces. A map $X \wedge Y \rightarrow Z$ is called a duality pairing if both \[\Sigma^\infty X \rightarrow F(\Sigma^\infty Y,\Sigma^\infty Z)\]  \[\Sigma^\infty Y \rightarrow F(\Sigma^\infty X,\Sigma^\infty Z)\] are equivalences.
\end{dfn}

\begin{dfn}
A loose framing of a smooth manifold $M$ is a section $F$ of \[\operatorname{Emb}([-\infty,\infty]^n,M) \xrightarrow{ev_0} M .\] Here, $\operatorname{Emb}(-,-)$ denotes smooth embeddings with the topology granted by the sup metric (for some metric on $M$ inducing the topology). By $F_x$ we denote the embedding $\mathbb{R}^n \rightarrow M$ given by $F(x)|_{\mathbb{R}^n}$.
\end{dfn}

We note that a smooth manifold $M$ is loosely framed, if and only if, $M$ has a framing, i.e. it has a trivialization of its tangent bundle. The forward direction follows from taking derivatives, and the backwards direction follows from taking the induced Riemannian metric and passing to a subbundle where it is injective on fibers.

\begin{thm}
Suppose $M$ is a smooth manifold with a loose framing $F$. Then the following is a duality pairing:

\[M_+ \wedge M^+ \xrightarrow{\chi_{F}} S^n\]
\[(x,y) \rightarrow F^{-1}_x(y)\]

If $y \not \in \operatorname{im}(F_x)$, we interpret $F_x^{-1}(y)$ as $\infty \in S^n$. 
\end{thm}

\begin{proof}

We first address the question of continuity of $\chi_{F}$ by analyzing the convergent sequences of $M_+ \wedge M^+$. This suffices since $M$ is a tame manifold, hence $M_+ \wedge M^+$ is first countable.

Suppose we have a sequence $(x_i,y_i) \rightarrow (x,\infty)$, then we must show the images $s_i \rightarrow \infty$. We know the embeddings $F_{x_i}$ converge to $F_x$ in the sup topology, so eventually $\operatorname{dist}(\mathrm{im}(F_{x_i}), \mathrm{im}(F_x))<\epsilon$ for $\epsilon$ arbitrarily small, and so there is a compact set $X$ eventually containing all the $\mathrm{im}(F_{x_i})$. Hence, the sequence $s_i$ is eventually constant at $\infty$.

A similar argument shows that if $(x_i,y_i) \rightarrow (x,y)$ with $y \not\in \mathrm{im}(F_x)$, then $s_i \rightarrow \infty$. This leaves $(x_i,y_i)\rightarrow s_i$ with $y \in \mathrm{im}(F_x)$. This implies the sequence $y_i$ is eventually contained in $\mathrm{im}(F_x)$. Then since $F_{x_i}$ converges to $F_x$ in the sup norm,  by increasing $i$ we can make the distance of $F^{-1}_{x_i}(y_i)$ and $F^{-1}_x(y)$ arbitrarily small since $y_i \rightarrow y$ . Since these are the three types of convergent sequences in $M_+ \wedge M^+$, our pairing is continuous.

We will show the adjoint of $\chi_{F}$ is an equivalence by showing that it induces isomorphisms in integral homology. This suffices since for finite spectra we have the homology Whitehead theorem: a map inducing isomorphism in integral homology is a weak equivalence. By \cite[Corollary 3A.7]{hatcher}, we can instead show homology isomorphisms for coefficients in $k=\mathbb{F}_p,\mathbb{Q}$.

To do this, we show that the pairing $H_*(M_+ ;k) \otimes H_*(M^+;k) \rightarrow H_*(S^n; k)$ is the Poincaré  duality pairing, since it is well known either adjoint is an isomorphism. For manifold representatives, the Poincaré  duality pairing may be computed by picking transverse representatives and counting the signed intersections. Goresky shows that this is true for arbitrary classes \cite[Section 3.6]{goresky_1981}.

Namely, if $a \in H_i(M_+ ; k), b \in H_{n-i}(M^+ ; k)$, there are transverse representatives $\alpha$ and $\beta$, in the sense that all intersections occur as the transverse intersection of top dimensional simplices, and by subdivision we can also assume there is at most one intersection per pair of complementary simplices. Further, the Poincaré  duality pairing is given by counting oriented intersections. There are two important consequences of this. First, if we consider the restriction of $\chi_{\bar{F}}$ to
 $\alpha_+ \wedge \beta \rightarrow S^n$, then, after possibly shrinking the loose framing, the pairing is cellular if we give $S^n$ the 2-cell structure with basepoint at $\infty$. This is because the intersections are in the interior of the top dimensional cells. Second, if we are given a pair of simplices 
 $d \in \alpha , e\in \beta$ of dimensions $i,n-i$ 
 respectively, then the  degree of the map $d \times e /\partial (d \times e) \rightarrow S^n$ at $0 \in S^n$ is either $0$, if there is no intersection of $d$ and $e$, or $\pm 1$, depending on the orientation of the intersection. This follows from the fact that a transverse intersection locally looks like $\mathbb{R}^i \cap \mathbb{R}^{n-i}$ inside $\mathbb{R}^n$, in which case one can verify the degree by hand. 

Hence, the image of $[\alpha] \otimes [\beta] $ in $H_n(S^n)$ is given by the sum over the signed intersections of the top dimensional simplices of $\alpha$ and $\beta$ which is exactly the signed intersection number.

\end{proof}

\section{Disk and sphere bundles}
In this section, we prove some results about disk and sphere bundles which will be required to study a twisted version of $(M^+)^{\wedge k}/ \Delta^{\mathrm{fat}}((M^+)^{\wedge k})$ in the case $M$ is not framed. The main point is that equivalences of bundles will imply suitably fiber preserving homotopy equivalences of objects constructed from them. All bundles will be assumed to be over paracompact spaces.

Our fiber bundles always have structure group $\mathrm{Homeo}(-)$. When possible, we name only the total spaces of our fiber bundles rather than the projection maps. Let $E \rightarrow X$ be a fiber bundle, for $x \in X$ we let $E|_x$ denote the fiber over $X$. If $y \in E$, we let $E_y$ denote the fiber that $y$ is contained in. If $E \rightarrow X$ is a closed disk bundle, we let $O(E)$  refer to the associated open disk bundle over $X$, and let $S(E)$  refer to the bounding sphere bundle. We abusively write $\bar{M}$ to refer to some fixed manifold which has $M$ as its interior.

We recall the notion of an equivalence of sphere bundles over different bases:

\begin{dfn}
Let $E_1 \rightarrow X_1,E_2 \rightarrow X_2$ be two sphere bundles. A homotopy equivalence of sphere bundles $f:E_1 \rightarrow E_2$ is a fiber preserving map $f:E_1 \rightarrow E_2$, so that there is a fiber preserving map $g:E_2 \rightarrow E_1$ with the property $f \circ g$ and $g \circ f$ are homotopic to the identity through fiber preserving maps.
\end{dfn}

We state some basic results about these objects with emphasis on how the particular homotopies can be chosen.

\begin{lem}
The pullback $f^*(E_2)$ of a sphere bundle $E_2 \rightarrow X_2$ along a homotopy equivalence $f:X_1 \rightarrow X_2$ is a homotopy equivalent sphere bundle, moreover, given a homotopy inverse $g$ of $f$ the homotopies witnessing the homotopy equivalence of sphere bundles can be chosen over any homotopies witnessing that $f$ and $g$ are inverse.
\end{lem}

\begin{proof}
There is an obvious map $F:f^*(E_2) \rightarrow E_2$. Suppose $g$ is a homotopy inverse of $f$. To construct the inverse map we must construct a map from $E_2|_x \rightarrow E_2|_{fg(x)}$. It is a classical fact (see \cite[Lemma 7.2]{mitch} for a proof) that a fiber bundle over $X \times I$ is isomorphic to the product of the interval with a fiber bundle over $X$. Applying this to the pullback of $E_2$ along the homotopy $H: Y \times I \rightarrow Y$ from $fg$ to the identity yields not only a map $E_2|_{x} \rightarrow E_2|_{fg(x)}$ sufficient to define an inverse $G$ of $F$, but a family of maps covering $H$ that demonstrates $G$ is a right homotopy inverse of $F$. Similarly, in the other direction we may pull back via the homotopy $fgf \sim f$ obtained from levelwise postcomposing any homotopy $gf \sim 1$ with $f$.
\end{proof}

We use the notation $X^+$ to denote the one point compactification of a space $X$. In the case $X$ is already compact, it is defined to be $X_+$, $X$ with a disjoint basepoint.

\begin{dfn}
An $n$-spherical fibration $E \rightarrow X$ is a fibration such that the fiber has the homotopy type of $S^n$.
\end{dfn}

\begin{dfn}
The Thom space $X^E$ of a spherical fibration $\rho:E \rightarrow X$ is $\operatorname{cofiber}(\rho)$.
\end{dfn}

\begin{prp}
If $E \rightarrow \bar{M}$ is a closed disk bundle, then \[(O(E)|_M)^+ \cong \bar{M}^{S(E)}/\partial\bar{M}^{S(E)|_{\partial \bar{M}}}.\]
\end{prp}

\begin{dfn}
Given $E_1 \rightarrow X_1, E_2 \rightarrow X_2$, we define a fiberwise map $F:E_1^+ \rightarrow E_2^+$ to be a pointed map so that if $x \in E_1$, then $(E_1)_x^+$ is mapped into $(E_2)_{F(x)}^+$.
\end{dfn}

Recall a proper map is a map where preimages of compact sets are compact. One point compactification is covariant with respect to such maps. The most obvious examples of fiberwise maps come from taking compactifications of proper bundles maps.

\begin{dfn}
A map $f:X \rightarrow Y$ is a proper homotopy equivalence if there exists a proper map $g:Y \rightarrow X$ such that $fg$ and $gf$ are homotopic to the identity through proper maps.
\end{dfn}

\begin{dfn}
Given $E_1 \rightarrow X_1, E_2 \rightarrow X_2$ a fiberwise homotopy equivalence $E_1^+ \sim E_2^+$ is a fiberwise map $f:E_1^+ \rightarrow E_2^+$ such that there is a fiberwise map $g:E_2^+ \rightarrow E_1^+$ with their compositions homotopic to the identity through fiberwise maps.
\end{dfn}

The most obvious examples of fiberwise homotopy equivalences come from homotopy equivalences of sphere bundles which cover proper homotopy equivalences of the bases, as in the following:

\begin{lem}
Suppose we have two closed disk bundles $E_1 \rightarrow X_1$ and $E_2 \rightarrow X_2$ and a homotopy equivalence of sphere bundles $F:S(E_1) \simeq S(E_2)$, with inverse $G$, covering a proper homotopy equivalence $f: X_1 \rightarrow X_2$ (with inverse $g$). Then $F$ induces a fiberwise homotopy equivalence $F^+:O(E_1)^+ \rightarrow O(E_2)^+$ with inverse $G^+$.
\end{lem}

\begin{proof}
By the Alexander trick, i.e. coning off our sphere bundle, we can extend our equivalence of sphere bundles to the disk bundles. By proceeding levelwise, we can extend the homotopies witnessing the equivalence of sphere bundles.  Passing to the open disk bundle, all these new maps will be levelwise proper, so levelwise one point compactifying the homotopies witnesses a fiberwise homotopy equivalence $F^+:E_1^+ \rightarrow E_2^+$ with inverse $G^+:E_2^+ \rightarrow E_1^+$.
\end{proof}

\section{Normal disk bundles}

In this section we review some results about normal disk bundles. A normal disk bundle of a manifold naturally ``untwists'' the manifold in the sense that the total space is always framed. By considering normal disk bundles, we may reduce duality questions to the framed case. These bundles have the remarkable property that they are suitably homotopy invariant, which is the key to our argument.

\begin{thm}
There exists a closed disk bundle $\hat{\nu} \rightarrow M$, so that $\hat{\nu}$ has a codimension $0$ embedding into Euclidean space.
\end{thm}

\begin{proof}
Assume $M$ is compact; the compact case implies the noncompact case by doubling $\bar{M}$ and restricting to $\operatorname{int}(\bar{M})$.

First embed $M$ into $\mathbb{R}^{N}$ in a locally flat manner \cite[Lemma 4.12]{kupers}. If this is high enough dimension the image is the zero section of a codimension 0 embedded $\mathbb{R}^d$ bundle \cite[Page 65]{hirsch}. After stabilizing once, we can ensure that there is a closed disk sub bundle (attributed to Mazur in \cite[Page 219]{browder}). We let this disk sub bundle be our closed disk bundle.
\end{proof}

\begin{dfn}
The normal open disk bundle is $\nu=O(\hat{\nu})$.
\end{dfn}

For the rest of this paper, we will reserve the symbol $\hat{\nu}$ to refer to such closed disk bundles and $\nu$ to refer to the corresponding open disk bundle. Call the rank of this bundle $c$.

The compact case of the following theorem is due to Spivak \cite[Proposition 5.6]{spivak}, and the general case is due to Wall \cite[Corollary 3.4]{wall}.

\begin{thm}[Spivak, Wall]
Up to homotopy, $\bar{M}^{S(\hat{\nu})} /\partial\bar{M}^{S(\hat{\nu})|_{\partial\bar{M}}}$ has a cell structure with a single $(n+c)$-cell, and this cell receives a degree $1$ map from $S^{n+c}$. Moreover, for sufficiently high and equal rank disk bundles with this property (not assumed to be from an embedding), the sphere bundle is unique up to fiberwise homotopy equivalence covering the identity.
\end{thm}

The following is well known in the compact case.

\begin{cor}
Assume we have a proper homotopy equivalence of manifolds $f:M_1 \simeq M_2$ and closed normal disk bundles $\hat{\nu}_1, \hat{\nu}_2$ of sufficiently high and equal rank for $M_1,M_2$, respectively. Then $f^*(S(\hat{\nu}_2))$ is homotopy equivalent as a sphere bundle to $S(\hat{\nu}_1)$, covering the identity. Equivalently, every proper homotopy equivalence $M_1 \simeq M_2$ can be covered by a homotopy equivalence of the sphere bundles $S(\hat{\nu}_1) \simeq S(\hat{\nu}_2)$. 
\end{cor}

\begin{proof}
Pullback $\hat{\nu}_2$ along our proper homotopy equivalence $f$. Since $M$ is the interior of $\bar{M}$, after changing this bundle by an isomorphism covering the identity, we can assume it extends to the boundary. More precisely, pick a deformation retraction $r:\bar{M}_1 \rightarrow \bar{M}_1 - \operatorname{im}c$ for a collar $c$ of $M_1$. We may pullback $f^*(\hat{\nu}_2)$ along $r$ and obtain a bundle over $\bar{M}_1$ which on $M_1$ is isomorphic to $f^*(\hat{\nu}_2)$ since $r$ is homotopic to the identity.

 We can now apply Proposition 3.5 combined with Lemma 3.2 to $O(f^*(\hat{\nu}_2))^+$ to deduce that $f^*(S(\hat{\nu}_2))$ satisfies the hypothesis of Theorem 4.3, from which the conclusion follows.
\end{proof}

An application of Lemma 3.9 yields:

\begin{cor}
Every proper homotopy equivalence $M_1 \simeq M_2$ induces a fiberwise homotopy equivalence $\nu_1^+ \sim \nu_2^+$.
\end{cor}

\section{Fat diagonals and configuration spaces}

In this section we define twisted versions of $(M^+)^{\wedge k} / \Delta^{\mathrm{fat}}((M^+)^{\wedge k})$ and show they are proper homotopy invariants. We then apply these results to prove the main theorem.

\begin{dfn}
For a pointed space $W$, the fat diagonal $\Delta^{\mathrm{fat}}(W^{\wedge k})$ is the subset of $W^{\wedge k}$ of points $(w_1,\dots,w_k)$ so that some $w_i=w_j$ for $i \neq j$. For an unpointed space $X$, let $\Delta^{\mathrm{fat}}(X^k)$ denote the analogous subspace of $X^k$.
\end{dfn}

\begin{dfn}
Given the $k$-fold external product $ E^k \rightarrow X^k$ of a fiber bundle $E \rightarrow X$, the bundle  $E^k_{\Delta^{\mathrm{fat}}}$ is the restriction of $E^k$ to $\Delta^{\mathrm{fat}}(X^k)$, i.e. $E^k|_{\Delta^{\mathrm{fat}}(X^k)}$
\end{dfn}

\begin{lem}
For open disk bundles $E_1 \rightarrow X_1,E_2 \rightarrow X_2$ a fiberwise homotopy equivalence $E_1^+ \sim E_2^+$ induces an $S_k$-equivariant homotopy equivalence of pairs $((E^k_1)^+,((E^k_1)_{\Delta^{\mathrm{fat}}})^+) \simeq ((E^k_2)^+,((E^k_2)_{\Delta^{\mathrm{fat}}})^+) $.
\end{lem}

\begin{proof}
After observing $(E^k_i)^+ \cong (E^+_i)^{\wedge k}$, this follows directly from the fact that the underlying map $(X_1^+)^{\wedge k} \rightarrow (X_2^+)^{\wedge k}$ preserves the fat diagonal as it is the $k$-fold smash product of a map $X_1^+ \rightarrow X_2^+$.
\end{proof}

\begin{dfn}
The configuration space of $k$ points in $M$, $F(M,k)$, is given by $M^k - \Delta^{\mathrm{fat}}(M^k)$.
\end{dfn}

Recall, $\nu$ denotes the open normal disk bundle of an embedding $M \rightarrow \mathbb{R}^{n+c}$.

\begin{dfn}
$\nu^k_{\mathrm{conf}}$ is the restriction of the external $k$-fold product $\nu^k$ to $F(M,k)$, i.e.  $\nu^k|_{F(M,k)}$.
\end{dfn}

\begin{thm}
There is an equivariant duality pairing

\[F(M,k)_+ \wedge (\nu^k_{\mathrm{conf}})^+ \rightarrow S^{(n+c)k}\]

where the group $S_k$ acts on everything in sight by permuting coordinates.
\end{thm}

\begin{proof}
We supply $\nu^k_{\mathrm{conf}}$ with a suitably equivariant loose framing which implies the result since $F(M,k) \simeq \nu^k_{\mathrm{conf}}$. Since $\nu$ embeds into $\mathbb{R}^{n+c}$, it is easily endowed with a loose framing $F$ by taking small cubical neighborhoods around each point. Taking products yields an equivariant loose framing $F^k$ on $\nu^k$. We describe how to restrict this loose framing to $\nu^k_{\mathrm{conf}}$. Given a point $x=(p_1,\dots,p_k,v_1,\dots,v_k)$ in $\nu^k_{\mathrm{conf}}$, identify the domain of $F^k_p$ which is $ ([-\infty,\infty]^{n+c})^k$ with $([-1,1]^{n+c})^k$. Let $r$ denote the supremum of $t \in(0,1)$ so that precomposing each factor of $[-1,1]^{n+c}$ by coordinate wise multiplication with $t$ projects to configuration space. Define the loose framing $F$ by precomposing with multiplication by $r/2$.  This loose framing has the property that when we permute the labels of the configuration, it simply permutes the factors of the loose framing. In other words, $F^{-1}_{gx}(gy) =gF^{-1}_x(y)$ where $g \in S_k$, but this is by definition what it means for the duality pairing to be equivariant.

\end{proof}

Let $M_1,M_2$ be proper homotopy equivalent manifolds. Let  $\nu_1:E_1\rightarrow M_1,$ \\$\nu_2:E_2 \rightarrow M_2$ be high and equal rank  open normal disk bundles.

\begin{thm}
There is a Borel equivariant equivalence \[\Sigma^\infty_+ F(M_1,k) \simeq \Sigma^\infty_+ F(M_2,k).\]
\end{thm}

\begin{proof}
For any $S_k$-spectrum $X$, the equivariant spectrum $F(X,(S^l)^{\wedge k})$, with the conjugation action, is a Borel equivariant homotopy invariant. So by the previous theorem it suffices to show that for $M_1 \simeq M_2$ properly homotopy equivalent, there is an equivariant equivalence $(\nu^k_{\mathrm{conf},1})^+ \simeq (\nu^k_{\mathrm{conf},2})^+$. We observe that $(\nu^k_{\mathrm{conf},i})^+ \cong (\nu^k_i)^+/(\nu^k_{\Delta^{\mathrm{fat}},i})^+$, so the result follows from applying Corollary 4.5 and Lemma 5.3.
\end{proof}

\section{Generalized configuration spaces}
Recently, variants of configuration spaces that allow some points to coincide have been investigated. Petersen and Bokstedt-Minuz have done initial cohomology calculations in what they call ``generalized configuration spaces" \cite{bm} \cite{petersen}, and Hai has given some combinatorial descriptions of ``partial configuration spaces" \cite{hai}. Additionally, Vakil and Wood study unordered versions of generalized configuration spaces in \cite{vw}.

In this section, we adapt our methods to study these generalized configuration spaces.

\begin{dfn}
A diagonal like formula $D$ is a formula in $x_1,\dots,x_k$ consisting of the following three terms with arbitrary parenthesization: equality, and, or. We say it is equivariant with respect to $H \leq S_k$ if any solution $y_1,\dots,y_k$ (meaning an ordered tuple of elements from $\{x_1,\dots,x_k\}$) gives another solution when acted upon by $H$.
\end{dfn}

Here are several examples: the formula $x_1=\dots=x_k$ is a diagonal like formula. This corresponds to the classical diagonal. Similarly, there is the collection of $x_i=x_j$ for $i \neq j$ separated by or's. This corresponds to the fat diagonal. These both are acted upon by $S_k$. Given an ordered partition $k_1 + \dots + k_l=k$, there is a diagonal like formula asserting that some $x_i=x_j$ for $i \neq j$, and $x_i$ and $x_j$ are in different blocks of the partition. This corresponds to the complement of the partial configuration spaces. This formula is acted upon by $S_{k_1} \times \dots \times S_{k_l}$.

From now on, we fix $D$ and the symmetry group $H$.

\begin{dfn}
For a pointed space $W$, the generalized fat diagonal $\Delta^D(W^{\wedge k})$ of $W^{\wedge k}$ is the subspace of points satisfying the formula $D$. For an unpointed space $X$, let $\Delta^D(X^k)$ denote the subspace of $X^k$ satisfying the formula $D$.
\end{dfn}

The fundamental property of $\Delta^D(-)$ is that a map of spaces $X \rightarrow Y$ induces a map of their $\Delta^D(-)$.

\begin{dfn}
Given the $k$-fold external product $ E^k \rightarrow X^k$ of an open disk bundle $E \rightarrow X$, the bundle  $E^k_{\Delta^{D}}$ is the  restriction of $E^k$ to the generalized fat diagonal, i.e. $E^k|_{\Delta^{D}(X^k)}$. 
\end{dfn}

\begin{lem}
For open disk bundles $E_1 \rightarrow X_1,E_2 \rightarrow X_2$ a fiberwise homotopy equivalence $E_1^+ \sim E_2^+$ induces an $H$-equivariant homotopy equivalence of pairs $((E^k_1)^+,((E^k_1)_{\Delta^D})^+) \simeq ((E^k_2)^+,((E^k_2)_{\Delta^D})^+) $.
\end{lem}

\begin{proof}
As before, this follows directly from the fact that the underlying map \\ $(X_1^+)^{\wedge k} \rightarrow (X_2^+)^{\wedge k}$ preserves the generalized fat diagonal as it is the $k$-fold smash product of a map $X_1^+ \rightarrow X_2^+$.
\end{proof}

\begin{dfn}
For a manifold $M$, the generalized D-configurations of $M$ are $F_D(M,k)=M^k - \Delta^D(M^k)$.
\end{dfn}

As before let $\nu$ denote the normal open disk bundle of an embedding $M \rightarrow \mathbb{R}^{n+c}$.
\begin{dfn}
The bundle $\nu^k_{\mathrm{conf(D)}}$ is the restriction of the $k$-fold external product of the bundle $\nu$ to $F_D(M,k)$, i.e. $\nu^k|_{F_D(M,k)}$.
\end{dfn}

Note that both the subspace $\Delta^D(M^k)$ and $F_D(M,k)$ are $H$-equivariant spaces if $D$ is $H$-equivariant.

\begin{thm}
Suppose $D$ is a diagonal like formula that is $H$-equivariant. Then there is an $H$-equivariant duality pairing,

\[F_D(M,k)_+ \wedge \nu_{\mathrm{conf}(D)}^+ \rightarrow S^{(n+c)k}\]

where the group $H$  acts on everything in sight by permuting coordinates.
\end{thm}

\begin{proof}
As before, we only need to define an appropriately $H$-equivariant framing on $\nu_{\mathrm{conf}(D)}$. Referring to the proof of theorem 6.6, instead of scaling our framing on $\nu^k$ to project to configuration space, we scale so it lands in $F_D(M,k)$.
\end{proof}

\begin{thm}
The $H$-equivariant spectrum $\Sigma^\infty_+ F_D(M,k)$ is a proper homotopy invariant.
\end{thm}

\begin{proof}
As before, it suffices to show that for $M_1 \simeq M_2$ properly homotopy equivalent, there is an equivariant equivalence $(\nu^k_{\mathrm{conf(D)},1})^+ \simeq (\nu^k_{\mathrm{conf(D)},2})^+$. We observe that $(\nu^k_{\mathrm{conf(D)},i})^+ \cong (\nu^k_i)^+/(\nu^k_{\Delta^{D},i})^+$, so the result follows from applying Corollary 4.5 and Lemma 6.4.
\end{proof}

\bibliographystyle{plain}
\bibliography{mybib}{}

\begin{thebibliography}{10}

\bibitem{ak}
Mokhtar Aouina and John~R Klein.
\newblock On the homotopy invariance of configuration spaces.
\newblock {\em Algebraic \& Geometric Topology}, 4(2):813–827, 2004.

\bibitem{browder}
William Browder.
\newblock Open and closed disk bundles.
\newblock {\em The Annals of Mathematics}, 83(2):218, 1966.

\bibitem{BCT}
C.~F. Bödigheimer, F.~Cohen, and L.~Taylor.
\newblock On the homology of configuration spaces.
\newblock {\em Topology}, 28(1):111–123, 1989.

\bibitem{bm}
Marcel Bökstedt and Erica Minuz.
\newblock Cohomology of generalised configuration spaces of points on
  $\mathbb{R}^r$, 2020.
\newblock \url{arXiv:2004.08370 }.

\bibitem{goresky_1981}
R.~Mark Goresky.
\newblock Whitney stratified chains and cochains.
\newblock {\em Transactions of the American Mathematical Society},
  267(1):175–196, 1981.

\bibitem{hai}
Li~Amy~Qing Hai.
\newblock {\em Partial configuration spaces as pullbacks of diagrams of
  configuration spaces}.
\newblock PhD thesis.

\bibitem{hatcher}
Allen Hatcher.
\newblock {\em Algebraic topology}.
\newblock Cambridge University Press, 2002.

\bibitem{hirsch}
M.~W. Hirsch.
\newblock On tubular neighborhoods of piecewise linear and topological
  manifolds.
\newblock {\em Conference on the topology of manifolds}, 1968.

\bibitem{knudsen}
Ben Knudsen.
\newblock Higher enveloping algebras.
\newblock {\em Geometry \& Topology}, 22(7):4013–4066, 2018.

\bibitem{kupers}
Alexander Kupers.
\newblock Three lectures on topological manifolds, 2017.
\newblock \url{https://people.math.harvard.edu/~kupers/notes/toplectures.pdf}.

\bibitem{LS}
Riccardo Longoni and Paolo Salvatore.
\newblock Configuration spaces are not homotopy invariant.
\newblock {\em Topology}, 44(2):375–380, 2005.

\bibitem{mitch}
Stephen Mitchell.
\newblock Notes on principal bundles and classifying spaces, 2011.
\newblock \url{https://math.mit.edu/~mbehrens/18.906spring10/prin.pdf}.

\bibitem{petersen}
Dan Petersen.
\newblock Cohomology of generalized configuration spaces.
\newblock {\em Compositio Mathematica}, 156(2):251–298, 2019.

\bibitem{spivak}
Michael Spivak.
\newblock Spaces satisfying {P}oincaré duality.
\newblock {\em Topology}, 6(1):77–101, 1967.

\bibitem{vw}
Ravi Vakil and Melanie~Matchett Wood.
\newblock Discriminants in the {G}rothendieck ring.
\newblock {\em Duke Mathematical Journal}, 164(6), 2015.

\bibitem{wall}
C.~T.~C. Wall.
\newblock Poincare complexes: I.
\newblock {\em The Annals of Mathematics}, 86(2):213, 1967.

\bibitem{zhang}
Adela~YiYu Zhang.
\newblock Quillen homology of spectral lie algebras with application to mod $p$
  homology of labeled configuration spaces, 2021.
\newblock \url{arXiv:2110.08428}.

\end{thebibliography}

\end{document}